\documentclass[10pt,a4paper,american]{amsart}
\usepackage{graphicx}%
\usepackage{multirow}%
\usepackage{amsmath,amssymb,amsfonts}%
\usepackage{amsthm}%
\usepackage{mathrsfs}%
\usepackage[title]{appendix}%
\usepackage{xcolor}%
\usepackage{textcomp}%
\usepackage{manyfoot}%
\usepackage{booktabs}%
\usepackage{algorithm}%
\usepackage{algorithmicx}%
\usepackage{algpseudocode}%
\usepackage{listings}%

\makeatletter

\numberwithin{equation}{section}
\numberwithin{figure}{section}

\usepackage{babel}
\usepackage{amsfonts}
\usepackage{amsthm}
\usepackage{mathcomp}
\usepackage{dsfont}
\usepackage{latexsym}

\def\Xint#1{\mathchoice
    {\XXint\displaystyle\textstyle{#1}}%
    {\XXint\textstyle\scriptstyle{#1}}%
    {\XXint\scriptstyle\scriptscriptstyle{#1}}%
    {\XXint\scriptscriptstyle\scriptscriptstyle{#1}}%
    \!\int}

\newtheorem{theo}{Theorem}[section]

\newtheorem{prop}[theo]{Proposition}
\newtheorem{lem}[theo]{Lemma}
\newtheorem{definition}[theo]{Definition}

\theoremstyle{definition}

\newtheorem*{lem*}{Lemma}
\newtheorem{rem}{Remark}

\newtheorem*{cor*}{Corollary}
\newtheorem*{theo*}{Theorem}

\newcommand{\R}{\ensuremath{\mathbb{R}}}

\newcommand{\BIGOP}[1]{\mathop{\mathchoice%
{\raise-0.22em\hbox{\huge $#1$}}%
{\raise-0.05em\hbox{\Large $#1$}}{\hbox{\large $#1$}}{#1}}}

\def\mwint_#1{\mathchoice
 {\mathop{\vrule width 6pt height 3 pt depth -2.5pt
        \kern -8.5pt \intop}\nolimits_{\kern -3pt #1}}%
 {\mathop{\vrule width 6pt height 3 pt depth -2.6pt
                  \kern -6pt \intop}\nolimits_{#1}}%
 {\mathop{\vrule width 6pt height 3 pt depth -2.6pt
                  \kern -6pt \intop}\nolimits_{#1}}%
 {\mathop{\vrule width 6pt height 3 pt depth -2.6pt
                  \kern -6pt \intop}\nolimits_{#1}}}

\renewcommand{\epsilon}{\varepsilon}

\numberwithin{equation}{section}

\def\Xint#1{\mathchoice
    {\XXint\displaystyle\textstyle{#1}}%
    {\XXint\textstyle\scriptstyle{#1}}%
    {\XXint\scriptstyle\scriptscriptstyle{#1}}%
    {\XXint\scriptscriptstyle\scriptscriptstyle{#1}}%
    \!\int}
\def\XXint#1#2#3{\setbox0=\hbox{$#1{#2#3}{\int}$}
    \vcenter{\hbox{$#2#3$}}\kern-0.5\wd0}

\def\dashint{\Xint{\raise4pt\hbox to7pt{\hrulefill}}}

\def\XXiint#1#2#3{\setbox0=\hbox{$#1{#2#3}{\iint}$}
    \vcenter{\hbox{$#2#3$}}\kern-0.5\wd0}

\allowdisplaybreaks

\makeatother

 \author[M. Eleuteri]{Michela Eleuteri}
 \address{Michela Eleuteri\\
 Dipartimento di Scienze Fisiche, Informatiche e Matematiche,
 Universit\`a degli Studi di Modena e Reggio Emilia, via Campi 213/b, 41125 -
 Modena, Italy
 }
 \email{michela.eleuteri@unimore.it}

 \author[S. Perrotta]{Stefania Perrotta}
 \address{Stefania Perrotta\\
 Dipartimento di Scienze Fisiche, Informatiche e Matematiche,
 Universit\`a degli Studi di Modena e Reggio Emilia, via Campi 213/b, 41125 -
 Modena, Italy
 }
 \email{stefania.perrotta@unimore.it}

 \author[G. Treu]{Giulia Treu}
 \address{Giulia Treu\\
 Dipartimento di Matematica ``Tullio Levi-Civita'', Universit\`a di Padova\\
 Via Trieste 63, 35121 Padova, Italy }
\email{giulia.treu@unipd.it}

\begin{document}
\title[Local Lipschitz continuity with lower order terms]{Local Lipschitz continuity for energy integrals with slow
growth and lower order terms}

%\author*[1]{\fnm{Michela} \sur{Eleuteri}}\email{michela.eleuteri@unimore.it}
%
%\author[1]{\fnm{Stefania} \sur{Perrotta}}\email{stefania.perrotta@unimore.it}
%
%
%\author[2]{\fnm{Giulia} \sur{Treu}}\email{giulia.treu@unipd.it}
%
%
%\affil*[1]{\orgdiv{Dipartimento di Scienze Fisiche, Informatiche e Matematiche}, \orgname{Universit\`a degli Studi di Modena e Reggio Emilia}, \orgaddress{\street{via Campi 213/b}, \city{Modena}, \postcode{41125}, %\state{State},
%\country{Italy}}}
%
%
%
%
%
%\affil[2]{\orgdiv{Dipartimento di Matematica `Tullio Levi-Civita'}, \orgname{Universit\`a di Padova}, \orgaddress{\street{Via Trieste 63}, \city{Padova}, \postcode35121}, %\state{State},
%\country{Italy}}

\begin{abstract}
{We consider integral functionals with slow growth and explicit dependence on $u$ of the lagrangian; this includes many relevant examples,  as, for instance, in elastoplastic torsion problems or in image restoration problems. Our aim is to prove that the local minimizers are locally Lipschitz continuous. The proof makes use of recent results concerning the  Bounded Slope Conditions.}
\end{abstract}


\maketitle

%\section{Introduction}

\section{Introduction and statement of the main result\label{Section
Introduction}}

Nowadays there is renewed interest regarding Lipschitz regularity results for local minimizers of integral functionals or weak solutions to a class of nonlinear
elliptic partial differential equations in divergence form with non-standard growth conditions, see for example the recent contributions \cite{BM, BS2022, CMMPdN, DFM2021, EMM, EMMP, EPdN, M2021}. Our paper fits into this research line, i.e. with the present paper our aim is to prove local Lipschitz regularity results for integral functionals of the type
\begin{equation}
\label{funz-modello}
\mathcal{F}(u)=\int_{\Omega} f(Du) + g(x, u) \, dx.
\end{equation} We emphasize our interest in dealing with the explicit dependence on $u$ of the lagrangian. This includes many significant functionals involved,  for instance, in elastoplastic torsion problems or in image restoration problems (see \cite{GT} for explicit examples).  
 Moreover this class of functionals has been already considered in literature, see for instance \cite{CCG, DFM2021} concerning
regularity of local minimizers of a class of integrals of the Calculus of Variations, see also \cite{DFM2022} where the functionals considered do not necessarily
satisfy the Euler–Lagrange equation. On the other hand, in \cite{M2022}  the motivation to introduce an explicit $u$-dependence on the coefficients in the differential equation comes from several recent studies on nonlinear elliptic and parabolic equations with general growth
conditions.

In this work we have been inspired by the papers \cite{mar96, EMMP} dealing with functionals depending only on $Du$  with general growth assumptions, respectively fast and slow. In both papers the authors prove suitable a priori estimates and then apply classical results on the Bounded Slope Condition to get the local Lipschitz continuity. We generalize these techniques in order to include also the lower order terms. To this aim we need to exploit some recent results that extend the classical ones on the Bounded Slope Conditions (BSC) to our type of functionals \cite{FT, GT}.

\bigskip

Our aim is to study the regularity of local minimizers of the functional \eqref{funz-modello} and, as it is commonly used in literature, we say that $u\in W_{loc}^{1,1}(\Omega )$ is a \textit{local minimizer} of the
integral functional $\mathcal{F}$ in \eqref{funz-modello} if $f(Du) + g(\cdot,u) \in
L_{loc}^{1}(\Omega )$ and 
\begin{equation*}
\int_{\Omega ^{\prime }}f(Du)+g(x,u)\,dx\leq \int_{\Omega ^{\prime }}f(Du+D\varphi
)+g(x,u+\varphi)\,dx
\end{equation*}%
for every open set $\Omega ^{\prime }$, $\overline{\Omega ^{\prime }}\subset
\Omega $ and for every $\varphi \in W_{0}^{1,1}(\Omega ^{\prime })$. 

\bigskip

In order to state the main result of the paper, we need to introduce the  assumptions on the lagrangian. Let $f:\mathbb{R}^{n}\rightarrow \lbrack 0\,,+\infty )$ be a
convex function in $\mathcal{C}(\mathbb{R}^{n})\cap \mathcal{C}^{2}(\mathbb{R%
}^{n}\backslash B_{t_{0}}(0))$ for some $t_{0}\geq 0$, satisfying the
following growth conditions: there exist two continuous functions $%
h_{1},h_{2}:[t_{0}\,,+\infty )\rightarrow (0\,,+\infty )$ and positive
constants $C_{1}$, $C_{2}$, $\alpha $, $\beta $ and $\mu\in[0\,,1] $ such
that 
\begin{enumerate}
    \item[(F1)] $ \displaystyle h_{1}\left( \left\vert \xi \right\vert \right) \left\vert \lambda
\right\vert ^{2}\leq \sum_{i,j=1}^{n}f_{\xi _{i}\xi _{j}}\left( \xi \right)
\lambda _{i}\lambda _{j}\leq h_{2}\left( \left\vert \xi \right\vert \right)
\left\vert \lambda \right\vert ^{2},\ \forall \lambda ,\xi \in 
\mathbb{R}^{n},\;\left\vert \xi \right\vert {\geq }t_{0}$;
\item[(F2)] $t\mapsto t^{\mu}h_{2}(t)$  is decreasing and $t\mapsto th_{1}(t)$ is
increasing;\\[0.05mm]
\item[(F3)] $\displaystyle \left( h_{2}(t)\right) ^{2/{2^*}}\leq C_{1}t^{2\beta }h_{1}(t)$,\;\;\;%
$\displaystyle \beta <\frac{2}{n},\;\forall \;t\geq t_{0}$;\\[0.08mm]
\item[(F4)] $h_{2}(\left\vert \xi \right\vert )|\xi |^{2}\leq C_{2}\,[1+f(\xi
)]^{\alpha },\;\;\alpha >1,\;\forall  \xi \in \mathbb{R}^{n}
,\;\left\vert \xi \right\vert \geq t_{0}$
\end{enumerate}

\vspace{2mm}

\noindent where, in (F3), $2^* = \frac{2n}{n-2}$ if $n \ge 3$ while in the case $n=2$, it must be replaced with any fixed positive number greater than $\frac{2}{1 - \beta}$. 

Moreover we assume that $g:\Omega\times\R\rightarrow\R$ is a Carathéodory function satisfying
\begin{enumerate}
    \item[(G1)] there exists $L$ such that $|g(x,\eta_1)-g(x,\eta_2)|\le L|\eta_1-\eta_2|$ for a.e. $x$ in $\Omega$ and every $\eta_1,\eta_2\in\R$; \vskip0.1cm
    \item[(G2)] $g(\cdot, 0) \in L^1_{loc}(\Omega);$ %there exists $\Lambda>0$ such that $\|g(\cdot,0)\|_{1}\le \Lambda$; 
    \vskip0.1cm
    \item[(G3)] $u\mapsto g(x,u)$ is convex for a.e. $x\in\Omega$;
    \vskip0.1cm
    \item[(G4)] there exists
a positive constant $K$ such that for a.e.  $x,y\in\R^n,\, \forall u,v\in\R$
\[
v\ge u+K|y-x|
\Rightarrow g^+_v(y,v)\ge g^+_v(x,u)
\]
where $g^+_v$ denotes the right derivative of
$g$ with respect to the second variable. 
\end{enumerate}

Condition (F2) means that we are considering {\it slow growth conditions.} Indeed in the model case of $p,q$-growth, the map $t \mapsto t^{\mu} h_2(t)$ is decreasing if and only if $q \le 2.$

It is moreover worth to highlight that we require uniform convexity and growth
assumptions on $f=f(\xi) $ only for large values of $\left\vert
\xi \right\vert $ (\cite{CE86},\cite{EMM},\cite{EMM20}).

\bigskip

We finally point out that the assumptions on $g$ are needed as far as we have to use the results on the Bounded Slope Condition for functionals depending also on $(x,u)$ and we refer to \cite{MTrado, FT, GT} for more details. We will discuss the meaning of this set of assumptions in Section \ref{cinque}.

\bigskip

The main 
result of the paper can be stated as follows. Notice that, here and in the sequel, we denote by $B_R$ a generic ball of radius $R$ compactly contained in $\Omega$ and by $B_{\rho}$ a
ball of radius $\rho < R$ concentric with $B_R.$

\begin{theo}
\label{superlinear} Let $u\in W_{loc}^{1,1}(\Omega)\cap L_{loc}^\infty(\Omega)$ be a local minimizer of the functional \eqref{funz-modello}. Suppose that $f$ satisfies the growth assumptions \textnormal{(F1)}--\textnormal{(F4)}, with the parameters $\alpha $, $\beta $, $\mu $ related by the
condition 
\begin{equation}
{2-\mu -\alpha (n\beta -\mu )}>0\,.  \label{ab}
\end{equation}%
Assume moreover that $g$ fulfills assumptions \textnormal{(G1)}--\textnormal{(G4)}.
\\
Then $u$  is locally Lipschitz continuous in $\Omega $ and  there exists $\bar R>0$ such that for every $0<\rho< R<\bar R$,  there exist two positive constants $C$ and $\kappa$ depending on the data of the problem, with $
C $ depending also on $\rho $, $R$  
and $\kappa$, depending also on $\|u\|_{L^{\infty}(B_R)}$, $\|g(\cdot, 0)\|_{L^1(B_R)}$ and $|B_R|,$ 
 such that 
\begin{equation}
\Vert Du\Vert _{L^{\infty }(B_{\rho }\,;\mathbb{R}^{n})}\leq \,C\,\left\{ 
\frac{1}{(R-\rho )^{n}} \left (\int_{B_{R}} f(Du) + g(x,u) \,dx+\kappa \right )\right\} ^{\theta }
\label{infinity nu}
\end{equation}%
where $\theta =\frac{(2-\mu )\alpha }{2-\mu -\alpha (n\beta -\mu )}$. 
\end{theo}

\vspace{2mm}

The paper is organized as follows. In Section \ref{due} are listed some notations and preliminary results. In particular we notice the use of the polar function of $f$ to construct suitable barriers for the minimizers. These barriers have been introduced by Cellina in \cite{C3} and satisfy a Comparison Principle \cite{FT}.  Section \ref{tre} contains the proof of the a priori estimate, namely Theorem \ref{stimapriori}; in this case we notice that condition \eqref{ab} is the same of \cite[Theorem 2.1]{EMMP} and it is not affected by the presence of the lower order term $g$. Section \ref{quattro} is devoted to the proof of Theorem \ref{superlinear}, which is divided in three steps: approximation, estimates for the boundedness of the gradient and passage to the limit. In this last part the crucial tools are the results concerning the (BSC) that still hold for a more general functional explicitly depending on $x$ and $u$. Finally Section \ref{cinque} aims to present some additional results related to Theorem \ref{superlinear}. In Theorem \ref{uLocBounded} we extend the main result to the case in which assumption (G1) is replaced by a local Lipschitz condition in the second variable. In Theorem \ref{uBounded} we present a regularity result for a Dirichlet problem in which it is not necessary to assume the a priori boundedness of the local minimizer. 
Finally,  Theorem \ref{radial-symmetric}  shows how, in the case the lagrangian depends on the modulus of the gradient, we allow slower growth conditions, more precisely condition \eqref{ab} is always satisfied. 

\bigskip

We conclude this introduction with some remarks showing a comparison between our results and the classical non-standard growth conditions.
\\
We begin with the widely studied case of $(p,q)-$growth, namely 
\[
h_1(t) \sim t^{p-2} \qquad h_2(t) \sim t^{q-2};
\]
in this case we obtain the usual gap condition
\[
\frac{q}{p} < 1 + \frac{2}{n}.
\]
On the other hand, if $f$ is strongly anisotropic as in \cite[Examples 3.3, 4.1 and 4.4]{EMMP}, then the regularity holds with the same conditions obtained for functionals depending only on the gradient

Moreover, in the radially symmetric case, we 
cover also the case of very slow growth such as
 $f(\xi )=(|\xi |+1)L_{k}(|\xi |)$, $%
k\in \mathbb{N}$, $L_{k}$ defined inductively as
\begin{equation*}
L_{1}\left( t\right) =\log \left( 1+t\right) ,\;\;\;\;\;L_{k+1}\left(
t\right) =\log \left( 1+L_{k}\left( t\right) \right),
\end{equation*}%
(see also \cite{FM00}). As a final remark we notice that, still in the radially symmetric case, for the $(p,q)$-growth, we obtain, as in \cite{MP}, the local Lipschitz regularity of the minimizers if 
\[
\frac{q}{p} < \frac{n}{n-2}. 
\]

\section{Notations and preliminary results}

\label{due}

Let $f:\R^n\rightarrow[0,+\infty)$ be a convex function. we denote $\partial f(x)$ the subdifferential of $f$ at the point $x$.
We indicate by $f^*$ the polar, or Fenchel tranform of $f$, see \cite{RW}, defined by
$$f^*(x):=\sup_{\xi\in \R^n}\{x\cdot\xi-f(\xi)\}, \quad \forall x\in \R^n.$$
We list here some properties that will be useful in the rest of the paper.
\begin{prop}
 \label{polare}
Let $f:\R^n\rightarrow[0,+\infty)$ be a convex function. Therefore:
 \begin{enumerate}
    \item[i)] if $f$ is superlinear,  then $f^*(x)\in \R$ for every $x\in\R^n$;
    \item[ii)] if $f$ is such that $f(\xi)\in\R$ for every $\xi \in\R^n$,  then $f^*(x)$ is superlinear;
    \item[iii)] $\partial f(\xi)=\left(\partial f^*\right)^{-1}(\xi)$ for every $\xi\in \R^n$;
    \item[iv)] if $f$ is superlinear $\partial f^*(x)=\left(\partial f\right)^{-1}(x)$ for every $x\in \R^n$; 
    \item[v)] if $f$ is in $\mathcal{C}(\mathbb{R}^{n})\cap \mathcal{C}^{2}(\mathbb{R%
}^{n}\backslash B_{t_{0}}(0))$ for some $t_{0}\geq 0$, superlinear and satisfies assumption {\rm (F1)}, then there exists $s_0\in\R$, such that $f^*\in\mathcal{C}^2(\R^n\setminus B_{s_0}(0))$. Moreover, for every $x\in\R^n\setminus B_{s_0}(0)$,  $D^2f^*(x)=\left(D^2f\right)^{-1}(Df^*(x))$ and 
\begin{equation}\label{ellitticitapolare}
\displaystyle\frac{1}{h_{2}\left( \left\vert Df^*(x) \right\vert \right) }\left\vert \lambda
\right\vert ^{2}\leq \sum_{i,j=1}^{n}f^*_{x _{i}x _{j}}\left( x \right)
\lambda _{i}\lambda _{j}\leq \frac{1}{h_{1}\left( \left\vert Df^*(x) \right\vert \right)}
\left\vert \lambda \right\vert ^{2},\ \forall \lambda ,x \in 
\mathbb{R}^{n},\;\left\vert x \right\vert {\geq }s_{0}.   
\end{equation}
\end{enumerate}   
\end{prop}
\begin{proof}
\ Statements i) and ii) follow from Lemma 3.1 in \cite{GT}, observing that it is not restrictive to assume that $f(0)=0$. Properties iii) and iv) are proved in \cite[Theorem 11.3]{RW} and v) is a consequence of \cite[Theorem 13.21]{RW} and subsequent observation.     
\end{proof}

\medskip

The next two lemmas will play an important role  in the third step of the proof of the main theorem.

\begin{lem}\label{comparison} Let $A$ be an open bounded subset of $\R^n$ with regular boundary. Fix  $\varphi\in L^\infty(A)$ and assume that $f:\R^n\rightarrow\R$ is convex and superlinear. Let $g:\Omega\times\R\rightarrow\R$ be a Carathéodory function satisfying \textnormal{(G1)}.
Let $u$ be a minimizer of 
\begin{equation*}
\int_{A} f(Du) + g(x, u) \, dx
\end{equation*}
in the class $\varphi+W_0^{1,1}(A)$. Then $u$ is essentially bounded on $A$.
\end{lem}

\begin{proof} 
$\ $ We start proving that there exists a constant $K^-$ such that $u(x)\ge K^-$ a.e. on $A$. 
We fix $x_0\in\R^n$  and we consider the function
\begin{equation}
\omega_L(x)= \frac{n}{L}f^*\left (\frac{L}{n}(x-x_0)\right )+\inf_{\partial A} \varphi(x)-\sup_{\partial A} \frac{n}{L}f^*\left (\frac{L}{n}(x-x_0) \right)   
\end{equation}
and we observe that, thanks to Proposition \ref{polare} i), it is well defined for every $x\in\R^n$. Moreover the definition of $\omega_L$ implies that
$$\omega_L(x)\le\varphi(x) \text{ on } \partial A$$ 
so that Theorem 2.4 in \cite{FT}, see also Theorem 2.4 in \cite{GT}, implies that then 
$$u(x)\ge K^-=\inf_A \omega_L(x)=\inf_{A}\frac{n}{L}f^*\left (\frac{L}{n}(x-x_0) \right )+\inf_{\partial A} \varphi(x)-\sup_{\partial A} \frac{n}{L}f^*\left (\frac{L}{n}(x-x_0) \right)$$
a.e. in $A$. The proof of the fact that there exists $K^+$ such that $u(x)\le K^+$ a.e. in $A$ follows in an analogous way using the function
\begin{equation}
\omega_{- L}(x)= -\frac{n}{L}f^*\left (-\frac{L}{n}(x-x_0)\right )+\sup_{\partial A} \varphi(x)-\inf_{\partial A} \left(- \frac{n}{L}f^*\left (-\frac{L}{n}(x-x_0) \right)   \right )
\end{equation}
and choosing
\begin{eqnarray*}
K^+&=&\sup_A \omega_{-L}(x)\\
&=&\sup_{A} \left (-\frac{n}{L}f^*\left(-\frac{L}{n}(x-x_0) \right)\right )+\sup_{\partial A} \varphi(x)-\inf_{\partial A}\left (- \frac{n}{L}f^*\left(-\frac{L}{n}(x-x_0)\right) \right ).
\end{eqnarray*}
\end{proof}

Before stating the next lemma, we need to recall the following classical definition.
\begin{definition}[BSC]\label{defbsc}
The function $\phi$ satisfies the \emph{Bounded Slope Condition} of
rank $m\ge 0$ if for every $\gamma\in\partial\Omega$ there exist
$z_{\gamma}^-$, $z_{\gamma}^+\in \R^n$ and $m\in\R$ such that
\begin{equation}\label{tag:lbsc} \forall\gamma'\in\partial\Omega\quad
\phi(\gamma)+z_{\gamma}^-\cdot(\gamma'-\gamma)\le\phi(\gamma')\le \phi(\gamma)+z_{\gamma}^+\cdot(\gamma'-\gamma)\end{equation}
%\begin{equation}\label{tag:ubsc} %\forall\gamma'\in\partial\Omega\quad
%\phi(\gamma)+z_{\gamma}^+\cdot(\gamma'-%\gamma)\ge\phi(\gamma')\end{equation}
and $|z_{\gamma}^{\pm}|\le m$ for every $\gamma\in\partial\Omega$.
\end{definition}

\begin{lem}\label{Lemmabound}Assume that $f_-,\tilde f,f_+:\R^n\rightarrow\R$ are convex and superlinear. Moreover assume that 
\begin{equation}\label{fmenopiù}
    f_-(\xi)\le \tilde f(\xi)\le f_+(\xi).
\end{equation}
Let $g:\Omega\times\R\rightarrow\R$  a Carathéodory function satisfying \textnormal{(G1)}. 
Let $\varphi$ satisfy the  
 {\rm (BSC)} and $\tilde u$ be a minimizer of 
\begin{equation}\label{funzftilde}
\int_{B_R} \tilde f(D v)+g(x,v)\, dx \qquad v\in\varphi+W^{1,1}_0(B_R).   \end{equation}
Then there exists $\bar K=\bar K(L,f_-,f_+,\|\varphi\|_{L^\infty(B_R)})$
such that $\|\tilde u\|_{L^{\infty}(B_R)}\le \bar K$. 
\end{lem}
\begin{proof}\ 
Assumption \eqref{fmenopiù}  implies that 
\begin{equation}
    f_+^*(\xi)\le\tilde f^*(\xi)\le f_-^*(\xi).
\end{equation}
As in the proof of Lemma \ref{comparison} we obtain
$$
\tilde u(x)\ge \frac{n}{L}\tilde f^* \left (\frac{L}{n}(x-x_0) \right )+\inf_{\partial B_R} \varphi(x)-\sup_{\partial B_R} \frac{n}{L}\tilde f^*\left (\frac{L}{n}(x-x_0) \right )
$$
$$
\ge \frac{n}{L}f_+^*\left (\frac{L}{n}(x-x_0)\right )-\|\varphi\|_{L^\infty(B_R)}-\sup_{\partial B_R} \frac{n}{L}f_-^*\left (\frac{L}{n}(x-x_0) \right ).
$$
Similar computation yields the inequality from above.
\end{proof}

\bigskip

\section{A priori estimates}

\label{tre}

In this section we prove a result that is in the same flavour of
\cite[Lemma 6.2]{EMMP}; we underline that here we deal also with lower order terms. For this reason in the proof
we will highlight only the main technical points that arise from the different structure of the functional.

\begin{theo}
\label{stimapriori}  Suppose that $f$ satisfies the growth assumptions \textnormal{(F1)}--\textnormal{(F4)} with the parameters $\alpha $, $\beta $, $\mu $ related by \eqref{ab}.
In addition, assume that $f$ is of class $\mathcal{C}^{2}(\mathbb{R}^{n})$ and for every $%
M>0$ there exists a positive constant $\ell =\ell (M)$ such that 
\begin{equation}
\ell \,|\lambda |^{2}\leq \,\sum_{i,j=1}^{n}f_{\xi _{i}\xi _{j}}(\xi
)\,\lambda _{i}\,\lambda _{j}\qquad \forall \lambda ,\xi \in \mathbb{R}%
^{n},\,|\xi |\leq M.  \label{supp1}
\end{equation}%

Assume moreover that $g$ fulfills assumptions \textnormal{(G1)}.
\\
Let $u\in W_{loc}^{1,\infty}(\Omega)$ be a local minimizer of the functional \eqref{funz-modello}.
Then  for every $R>0$ sufficiently small and
$
0<\rho <R$ there exists a positive constant $
C $ depending on $\rho $, $R$, $C_{1}$, $C_{2}$, $\alpha $, $\beta $, $\mu $
, $h_{1}(t_{0})$, $L$, such that 
\begin{equation}
\label{PP2}
\Vert Du\Vert _{L^{\infty }(B_{\rho }\,;\mathbb{R}^{n})}\leq \,C\,\left\{ 
\frac{1}{(R-\rho )^{n}}\int_{B_{R}}\{1+f(Du)\}\,dx\right\} ^{\theta }
\end{equation}%
where $\theta =\frac{(2-\mu )\alpha }{2-\mu -\alpha (n\beta -\mu )}$.
\end{theo}

Before proving this result, we state a Lemma, analogous to Lemma 6.1 in \cite{EMMP}, where we suppose that  $t \mapsto t h_1(t)$ is increasing while in in \cite{EMMP} it is assumed that
$t \mapsto t h_2(t)$ is increasing.
The proof is omitted since it requires only minor changes. For the sake of simplicity we assume, in the remaining part of the section, that $t_0=1$.

\begin{lem}
\label{g1eg2} Let us assume that \textnormal{(F2)} and \textnormal{(F3)} hold. Then
for every $\gamma\geq 0$ there exists a constant $C_3= C_3(C_1\,,h_1(1))>0$
independent of $\gamma$, such that, for every $t\geq 0$
\begin{equation}  \label{eq:int1}
C_3\left[ 1+ h_2(1 + t)^{\frac{1}{2^*}} \frac{(1 + t)^{\frac{\gamma}{2}%
+1-\beta}}{ \left (\frac{\gamma}{2} + 1 - \beta \right )^2} \right] \leq
1+\int_0^t (1+s)^{\frac{\gamma - 2}{2}} s \sqrt{h_1(1+s)}\,ds.
\end{equation}
%, where , for $n>2$, $2^{\ast }=\frac{2n}{n-2}$, while,
%for $n=2$, $2^{\ast }$ can be any number greater than $\frac 2{1-\beta}$.
\end{lem}

Now we proceed with the proof of Theorem \ref{stimapriori}.

\begin{proof}\
Since the local minimizer $u$ is in $W_{loc}^{1,\infty }(\Omega )$,
it satisfies the Euler equation: for every open set $\Omega ^{\prime }$
compactly contained in $\Omega $ we have 
\begin{equation*}
\int_{\Omega }\sum_{i=1}^{n}f_{\xi _{i}}(Du)\,\varphi _{x_{i}} + g_u(x,u) \, \varphi \,dx=0\qquad
\forall \varphi \in W_{0}^{1,2}(\Omega ^{\prime }).
\end{equation*}%
Moreover, by the techniques of the difference quotient (see for example \cite%
[Theorem $1.1'$, Ch. II]{giaquinta}), $u\in W_{\mathrm{loc}}^{2,2}(\Omega )$, then
the second variation holds: 
\begin{equation}\label{varseconda}
\int_{\Omega }\sum_{i,j=1}^{n}f_{\xi _{i}\xi _{j}}(Du)u_{x_{j}x_{k}}\varphi
_{x_{i}} - g_u(x,u) \varphi_{x_k} \,dx=0,\quad \forall k=1,\dots ,n
\end{equation}
$\forall \varphi \in
W_{0}^{1,2}(\Omega ^{\prime })$. 

For fixed $k=1,\dots ,n$ let $\eta \in \mathcal{C}_{0}^{1}(\Omega ^{\prime
}) $ be equal to 1 in $B_{\rho }$, with support contained in $B_{R}$, such
that $|D\eta |\leq \,\frac{2}{(R-\rho )},$ and consider
\[
\varphi =\eta^{2}\,u_{x_{k}}\,\Phi ((|Du|-1)_{+})
\]
with $\Phi $ non negative, increasing,
locally Lipschitz continuous on $[0,+\infty )$, such that $\Phi (0)=0$. Here $(a)_{+}$ denotes the positive part of $a\in \mathbb{R}$; in the following we denote $\Phi ((|Du|-1)_{+})=\Phi (|Du|-1)_{+}.$
Then a.e. in $\Omega$ 
\begin{eqnarray*}
\varphi _{x_{i}}&=& 2\eta \,\eta _{x_{i}}u_{x_{k}}\Phi (|Du|-1)_{+}+\eta
^{2}u_{x_{i}x_{k}}\Phi (|Du|-1)_{+}\\
&& +\eta ^{2}u_{x_{k}}\Phi ^{\prime
}(|Du|-1)_{+}[(|Du|-1)_{+}]_{x_{i}}.
\end{eqnarray*}%
Proceeding along the lines of \cite{mar96}, see also \cite{EMMP}, we therefore deduce that 
\begin{eqnarray*}
0 &=&\int_{\Omega }2\eta \Phi (|Du|-1)_{+} \sum_{i,j=1}^{n} \eta _{x_{i}}u_{x_{k}} f_{\xi _{i}\xi
_{j}}(Du)u_{x_{j}x_{k}} \,dx \\
&& -\int_{\Omega }2\eta \Phi (|Du|-1)_{+}  \eta _{x_{k}}u_{x_{k}} g_u(x,u) \,dx \\
&&+\int_{\Omega }\eta ^{2}\Phi (|Du|-1)_{+}\sum_{i,j=1}^{n}f_{\xi _{i}\xi
_{j}}(Du)u_{x_{j}x_{k}}u_{x_{i}x_{k}}\,dx \\
&&- \int_{\Omega }\eta^2 \Phi (|Du|-1)_{+} u_{x_{k} x_k} g_u(x,u) \,dx \\
&&+\int_{\Omega }\eta ^{2}\Phi ^{\prime }(|Du|-1)_{+}\sum_{i,j=1}^{n}f_{\xi
_{i}\xi _{j}}(Du)u_{x_{j}x_{k}}u_{x_{k}}[(|Du|-1)_{+}]_{x_{i}}\,dx \\
&&-\int_{\Omega }\eta ^{2}\Phi ^{\prime }(|Du|-1)_{+} g_u(x,u) u_{x_{k}}[(|Du|-1)_{+}]_{x_{k}}\,dx \\
&=& I_{1k} - I_{2k} + I_{3k} - I_{4k} + I_{5k} - I_{6k}. 
\end{eqnarray*}%
We now sum the previous equation with respect to $k$ from $1$ to $n$, and we denote by $I_1$ - $I_6$
the corresponding integrals. 

First of all we have
\begin{equation}
\label{elenco}
I_3 + I_5 \le \, |I_1| + |I_2| + |I_4| + |I_6|.
\end{equation}
We start by estimating $|I_{1}|$ by using the
Cauchy-Schwarz inequality and the Young inequality so that 
\begin{equation*}
\begin{split}
|I_{1}| & = \left\vert \int_{\Omega }2\eta \Phi (|Du|-1)_{+}\sum_{i,j, k=1}^{n}f_{\xi
_{i}\xi _{j}}(Du)u_{x_{j}x_{k}}\eta _{x_{i}}u_{x_{k}}\,dx\right\vert \\
& \quad \leq \int_{\Omega }2\Phi (|Du|-1)_{+}\left( \eta
^{2}\sum_{i,j, k=1}^{n}f_{\xi _{i}\xi
_{j}}(Du)u_{x_{i}x_{k}}u_{x_{j}x_{k}}\right) ^{\frac{1}{2}} \\
& \qquad \times \left(
\sum_{i,j, k=1}^{n}f_{\xi _{i}\xi _{j}}(Du)\eta _{x_{i}}u_{x_{k}}\eta
_{x_{j}}u_{x_{k}}\right) ^{\frac{1}{2}}\,dx \\
& \quad \leq \frac{1}{2}\int_{\Omega }\eta ^{2}\Phi
(|Du|-1)_{+}\sum_{i,j, k=1}^{n}f_{\xi _{i}\xi
_{j}}(Du)u_{x_{i}x_{k}}u_{x_{j}x_{k}}\,dx \\
& \qquad +2\int_{\Omega }\Phi (|Du|-1)_{+}\sum_{i,j, k=1}^{n}f_{\xi _{i}\xi
_{j}}(Du)\eta _{x_{i}}u_{x_{k}}\eta _{x_{j}}u_{x_{k}}\,dx\\
& \quad \leq \frac{1}{2}\int_{\Omega }\eta ^{2}\Phi
(|Du|-1)_{+}\sum_{i,j, k=1}^{n}f_{\xi _{i}\xi
_{j}}(Du)u_{x_{i}x_{k}}u_{x_{j}x_{k}}\,dx \\
& \qquad + 2 \int_{\Omega } |D \eta|^2 \, \Phi (|Du|-1)_{+} h_2(1 + (|Du|-1)_+) \, (1 + (|Du|-1)_+)^2 \,dx.
\end{split}%
\end{equation*}%
On the other hand, by (F2), recalling that, for $t\ge 1$, $h_2(t) \ge h_1(t) \ge \, {h_1(1)}/t$, we deduce
\begin{equation*}
\begin{split}
|I_{2}| & = \left\vert \int_{\Omega }  2\eta \Phi (|Du|-1)_{+} \sum_{k=1}^n \eta _{x_{k}}u_{x_{k}} g_u(x,u) \,dx
\right\vert \\
& \leq \, L  \int_{\Omega} (\eta^2 + |D \eta|^2) |Du| \,  \Phi (|Du|-1)_{+} \, dx\\
& \leq \, \frac{L}{h_1(1)} \int_{\Omega} (\eta^2 + |D \eta|^2) \Phi (|Du|-1)_{+} h_2(|Du|) \, |Du|^2 \, dx\\
& = \frac{L}{h_1(1)} \int_{\Omega} (\eta^2 + |D \eta|^2) \Phi (|Du|-1)_{+} h_2(1 + (|Du|-1)_+) \, (1 + (|Du|-1)_+)^2  \, dx.
\end{split}%
\end{equation*}%
%Therefore we deduce 
%\begin{equation*}
%\begin{split}
%\frac{1}{2}\int_{\Omega }\eta ^{2}\Phi (|Du|-1)_{+}& \sum_{i,j=1}^{n}f_{\xi
%_{i}\xi _{j}}(Du)u_{x_{i}x_{k}}u_{x_{j}x_{k}}\,dx \\
%& +\int_{\Omega }\eta ^{2}\Phi ^{\prime }(|Du|-1)_{+}\sum_{i,j=1}^{n}f_{\xi
%_{i}\xi _{j}}(Du)u_{x_{k}}[(|Du-1|)_{+}]_{x_{i}}\,dx \\
%& \leq 2\int_{\Omega }\Phi %(|Du|-1)_{+}\sum_{i,j=1}^{n}f_{\xi _{i}\xi
%_{j}}(Du)\eta _{x_{i}}u_{x_{k}}\eta %_{x_{j}}u_{x_{k}}\,dx.
%\end{split}%
%\end{equation*}%
Now we estimate the term $|I_4|.$ Taking into account that, for $t\ge 1$, 
\begin{equation}
    \label{g1g2}
 h_1(t) \, h_2(t) \, t^2\ge h_1(t)^2t^2 \ge h_1(1)^2
\end{equation}
we have
\begin{equation*}
\begin{split}
|I_{4}| & = \left\vert \int_{\Omega }  \eta^2 \Phi (|Du|-1)_{+} \sum_{k=1}^n u_{x_{k} x_k}  g_u(x,u) \,dx
\right\vert \\
& \leq \, \frac{nL}{h_1(1)} \int_{\Omega} \left [\eta^2 |D^2u|^{2}  \Phi (|Du|-1)_{+} h_1(1 + (|Du| - 1)_+)\right ]^{\frac{1}{2}} \\
& \quad \times \left [\eta^2 \Phi (|Du|-1)_{+} h_2(1 + (|Du| - 1)_+) (1 + (|Du|-1)_+))^2\right ]^{\frac{1}{2}} \, dx\\
&  \leq \, \varepsilon \, \int_{\Omega} \eta^2  \Phi (|Du|-1)_{+} h_1(1 + (|Du| - 1)_+) |D^2 u|^{2} \, dx \\
& \quad + \frac{n^2 L^2}{4 h_1(1)^2 \varepsilon}  \int_{\Omega} \eta^2 \Phi (|Du|-1)_{+} h_2(1 + (|Du| - 1)_+) (1 + (|Du|-1)_+)^2 \, dx,
\end{split}%
\end{equation*}%
where $\varepsilon$ is a positive parameter that will be suitably chosen later.
\\
We then estimate $|I_6|$ as follows
\begin{equation*}
\begin{split}
|I_{6}| & = \left\vert \int_{\Omega }\eta ^{2}\Phi ^{\prime }(|Du|-1)_{+} g_u(x,u) \sum_{k=1}^n u_{x_{k}}[(|Du|-1)_{+}]_{x_{k}}\,dx
\right\vert \\
& \leq  \, L \int_{\Omega} \eta ^{2}\Phi ^{\prime }(|Du|-1)_{+} (1 + (|Du|-1)_+) |D(|Du|-1)_{+}|\,dx\\
& \stackrel{\eqref{g1g2}}{\leq} \, \frac{L}{h_1(1)} \int_{\Omega} \eta ^{2}\Phi ^{\prime }(|Du|-1)_{+} (1 + (|Du|-1)_+) |D(|Du|-1)_{+}| \\ 
& \quad \times \sqrt{h_1(1 + (|Du|-1)_+) \, h_2(1 + (|Du| -1)_+)} (1 + (|Du| - 1)_+)\,dx\\
& \leq  \, \frac{L}{h_1(1)} \int_{\Omega} \Big [\eta^2 |D(|Du|-1)_{+}|^{2} \Phi ^{\prime }(|Du|-1)_{+} (1 + (|Du|-1)_+)\\
& \quad \times h_1(1 + (|Du| - 1)_+)\Big ]^{\frac{1}{2}} \Big [\eta^2 \Phi ^{\prime }(|Du|-1)_{+} (1 + (|Du|-1)_+) \\
& \times h_2(1 + (|Du| - 1)_+) (1 + (|Du|-1)_+)^2\Big ]^{\frac{1}{2}} \, dx\\
&  \leq  \varepsilon \, \int_{\Omega} \eta^2  \Phi ^{\prime }(|Du|-1)_{+} (1 + (|Du|-1)_+) h_1(1 + (|Du| - 1)_+) |D(|Du|-1)_{+}|^{2} \, dx \\
& \quad + \, \frac{L^2}{4 h(1)^2 \varepsilon} \int_{\Omega} \eta^2 \Phi ^{\prime }(|Du|-1)_{+} (1 + (|Du|-1)_+) \\
&\times h_2(1 + (|Du| - 1)_+) (1 + (|Du|-1)_+)^2 \, dx. 
\end{split}%
\end{equation*}%
Finally, since a.e.~in $\Omega $ 
\begin{equation*}
\lbrack (|Du|-1)_{+}]_{x_{i}}=%
\begin{cases}
(|Du|)_{x_{i}}=\frac{1}{|Du|}\sum_{k}u_{x_{i}x_{k}}u_{x_{k}} & \text{ if $%
|Du|>1$,} \\ 
0 & \text{ if $|Du|\leq 1$,}%
\end{cases}%
\end{equation*}%
we obtain 
\begin{equation*}
\begin{split}
I_5 = & \int_{\Omega} \eta^2 \Phi'(|Du|-1)_{+} \sum_{i,j,k=1}^{n} f_{\xi _{i}\xi
_{j}}(Du)u_{x_{j}x_{k}}u_{x_{k}}[(|Du|-1)_{+}]_{x_{i}} \, dx \\
= & \int_{\Omega} \eta^2 \Phi'(|Du|-1)_{+} |Du|%
\sum_{i,j=1}^{n}f_{\xi _{i}\xi
_{j}}(Du)[(|Du-1|)_{+}]_{x_{j}}[(|Du|-1)_{+}]_{x_{i}} \, dx.
\end{split}
\end{equation*}%
%therefore we deduce the estimate 
%\begin{equation*}
%\begin{split}
%\int_{\Omega }\eta ^{2}& \Phi (|Du|-1)_{+}\sum_{k,i,j=1}^{n}f_{\xi _{i}\xi
%_{j}}(Du)u_{x_{j}x_{k}}u_{x_{i}x_{k}}\,dx \\
%& +\int_{\Omega }\eta ^{2}|Du|\Phi ^{\prime
%}(|Du|-1)_{+}\sum_{i,j=1}^{n}f_{\xi _{i}\xi
%_{j}}(Du)[(|Du-1|)_{+}]_{x_{j}}[(|Du-1|)_{+}]_{x_{i}}\,dx \\
%& \leq 4\int_{\Omega }\Phi %(|Du|-1)_{+}\sum_{k,i,j=1}^{n}f_{\xi _{i}\xi
%_{j}}(Du)\eta _{x_{i}}u_{x_{k}}\eta %_{x_{j}}u_{x_{k}}\,dx.
%\end{split}%
%\end{equation*}%
Inserting the estimates obtained in \eqref{elenco}, we deduce
\begin{equation*}
\begin{split}
& \int_{\Omega }\eta ^{2}\Phi (|Du|-1)_{+} \sum_{i,j,k=1}^{n} f_{\xi _{i}\xi
_{j}}(Du)u_{x_{j}x_{k}}u_{x_{i}x_{k}}\,dx \\
& + \int_{\Omega} \eta^2 \Phi'(|Du|-1)_{+} |Du| \sum_{i,j =1}^{n} f_{\xi _{i}\xi
_{j}}(Du)[(|Du-1|)_{+}]_{x_{j}}[(|Du|-1)_{+}]_{x_{i}} \, dx\\
\le & \frac{1}{2}\int_{\Omega }\eta ^{2}\Phi
(|Du|-1)_{+}\sum_{i,j, k=1}^{n}f_{\xi _{i}\xi
_{j}}(Du)u_{x_{i}x_{k}}u_{x_{j}x_{k}}\,dx \\
& + 2 \int_{\Omega } |D \eta|^2 \, \Phi (|Du|-1)_{+} h_2(1 + (|Du|-1)_+) \, (1 + (|Du|-1)_+)^2 \,dx\\
& + \frac{L}{h_1(1)} \int_{\Omega} (\eta^2 + |D \eta|^2) \Phi (|Du|-1)_{+} h_2(1 + (|Du|-1)_+) \, (1 + (|Du|-1)_+)^2  \, dx\\
& + \varepsilon \, \int_{\Omega} \eta^2  \Phi (|Du|-1)_{+} h_1(1 + (|Du| - 1)_+) |D^2 u|^{2} \, dx \\
& + \frac{L^2}{4 h(1)^2 \varepsilon}   \int_{\Omega} \eta^2 \Phi (|Du|-1)_{+} h_2(1 + (|Du| - 1)_+) (1 + (|Du|-1)_+)^2 \, dx \\
& +  \varepsilon \, \int_{\Omega} \eta^2  \Phi ^{\prime }(|Du|-1)_{+} (1 + (|Du|-1)_+) h_1(1 + (|Du| - 1)_+) |D(|Du|-1)_{+}|^{2} \, dx 
\\
& + \frac{L^2}{4 h(1)^2 \varepsilon}  \int_{\Omega} \eta^2 \Phi ^{\prime }(|Du|-1)_{+} (1 + (|Du|-1)_+)
\\
&\times h_2(1 + (|Du| - 1)_+) (1 + (|Du|-1)_+)^2 \, dx \\
\end{split}
\end{equation*}
Absorbing the first term in the right side of the inequality by the left hand side and rearranging the terms in the right hand side, we deduce
\begin{equation*}
\begin{split}
& \frac{1}{2} \int_{\Omega }\eta ^{2}\Phi (|Du|-1)_{+} \sum_{i,j, k=1}^{n} f_{\xi _{i}\xi
_{j}}(Du)u_{x_{j}x_{k}}u_{x_{i}x_{k}}\,dx \\
& + \int_{\Omega} \eta^2 \Phi'(|Du|-1)_{+} |Du| \sum_{i,j=1}^{n} f_{\xi _{i}\xi
_{j}}(Du)[(|Du|-1)_{+}]_{x_{j}}[(|Du|-1)_{+}]_{x_{i}} \, dx
\end{split}
\end{equation*}
\newpage
\begin{equation*}
\begin{split}
\le & \varepsilon \, \int_{\Omega} \eta^2  \Phi (|Du|-1)_{+} h_1(1 + (|Du| - 1)_+) |D^2 u|^{2} \, dx \\
& + \varepsilon \, \int_{\Omega} \eta^2  \Phi ^{\prime }(|Du|-1)_{+} (1 + (|Du|-1)_+)  h_1(1 + (|Du| - 1)_+) |D(|Du|-1)_{+}|^{2} \, dx \\
& + \frac{C}{\varepsilon}  \int_{\Omega} (\eta^2 + |D \eta|^2)  \Phi (|Du|-1)_{+}  h_2(1 + (|Du| - 1)_+) (1 + (|Du|-1)_+)^2 \, dx \\
& + \frac{C}{\varepsilon}  \int_{\Omega} (\eta^2 + |D \eta|^2)   \Phi ^{\prime }(|Du|-1)_{+} (1 + (|Du|-1)_+)\\
&\times h_2(1 + (|Du| - 1)_+) (1 + (|Du|-1)_+)^2 \, dx 
\end{split}
\end{equation*}
where $C$ is a constant depending only on $n, L, h_1(1).$ In the sequel we denote by $C$ a constant depending only on $n, L, h_1(1)$, not necessarily the same constant.
Using the ellipticity condition in (F1) and the inequality $|D(|Du|-1)_{+}|^{2}\leq \,|D^{2}u|^{2}$, choosing $\varepsilon$ sufficiently small, we then obtain 
\begin{equation}
\begin{split}
& \int_{\Omega } \eta ^{2} \Phi (|Du|-1)_{+}
h_{1}(1+(|Du|-1)_{+})\,|D(|Du|-1)_{+}|^{2}\,dx \\
\le & \int_{\Omega }\eta ^{2}[\Phi (|Du|-1)_{+}+|Du|\Phi ^{\prime }(|Du|-1)_{+}]{%
h_{1}(1+(|Du|-1)_{+}))}\,|D(|Du|-1)_{+}|^{2}\,dx \\
\leq & C \, \int_{\Omega} (\eta^2 + |D \eta|^2)  \Phi (|Du|-1)_{+}  h_2(1 + (|Du| - 1)_+) (1 + (|Du|-1)_+)^2 \, dx \\
& + C \, \int_{\Omega} (\eta^2 + |D \eta|^2)   \Phi ^{\prime }(|Du|-1)_{+} (1 + (|Du|-1)_+) \\
&\times h_2(1 + (|Du| - 1)_+) (1 + (|Du|-1)_+)^2 \, dx 
\end{split}
\label{(25)Jota}
\end{equation}%

Let us define 
\begin{equation}
G(t)=1+\int_{0}^{t}\sqrt{\Phi (s)\,h_{1}(1+s)}\,ds\qquad \forall t\geq \,0.
\label{defG}
\end{equation}%
By Jensen's inequality and the monotonicity of $\Phi $, since $t\mapsto
th_{1}(t)$ is increasing,
\begin{equation*}
\begin{split}
G(t)& =1+\int_{0}^{t}\sqrt{\Phi (s)(1+s)h_{1}(1+s)\frac{1}{1+s}}\,ds\\
%\leq
%1+\int_{0}^{t}\sqrt{\Phi (s)(1+s)h_{2}(1+s)\frac{1}{1+s}}\,ds \\
& \leq 1+\sqrt{\Phi (t)(1+t)h_{1}(1+t)}\int_{0}^{t}\frac{1}{\sqrt{1+s}}%
\,ds
\\
&\leq 1+2\sqrt{\Phi (t)(1+t)h_{1}(1+t)}\,\sqrt{1+t},
\end{split}%
\end{equation*}%
hence, recalling that $h_1 \le h_2$
\[
[G(t)]^{2}\leq 8\left[ 1+\Phi (t)(1+t)^{2}h_{1}(1+t)\right] \leq 8\left[ 1+\Phi (t)(1+t)^{2}h_{2}(1+t)\right]. 
\] 
On the
other hand 
\begin{equation*}
\begin{split}
|D& [\eta \, G(|Du|-1)_{+}]|^{2} \\
\leq & \,2\,|D\eta |^{2}[G((|Du|-1)_{+})]^{2}+2\eta ^{2}[G^{\prime
}((|Du|-1)_{+})]^{2}\,|D((|Du|-1)_{+})|^{2} \\
\leq & 16\,|D\eta |^{2}\, \left [1+\Phi (|Du|-1)_{+}\,h_{2}(1+(|Du|-1)_{+})(1+(|Du|-1)_{+})^{2} \right ]\\
& +2\,\eta
^{2}\,\Phi (|Du|-1)_{+}h_{1}(1+(|Du|-1)_{+})\,|D(|Du|-1)_{+}|^{2}.
\end{split}%
\end{equation*}%
Since $\Phi (|Du\left( x\right) |-1)_{+}=0$ when $|Du\left( x\right) |\leq 1$%
, by \eqref{(25)Jota} we get 
\begin{equation}
\begin{split}
& \int_{\Omega }|D(\eta \,G((|Du|-1)_{+})|^{2}\,dx \\
%& \leq \,24\int_{\Omega }|D\eta |^{2}\,(1+\Phi
%(|Du|-1)_{+}h_{2}(|Du|)|Du|^{2})\,dx \\
\le & C \, \int_{\Omega} (\eta^2 + |D \eta|^2)  \left [1 + \Phi (|Du|-1)_{+}  h_2(1 + (|Du| - 1)_+) (1 + (|Du|-1)_+)^2\right ] \, dx \\
& + C \, \int_{\Omega} (\eta^2 + |D \eta|^2)   \Phi ^{\prime }(|Du|-1)_{+} (1 + (|Du|-1)_+) \\
&\times h_2(1 + (|Du| - 1)_+) (1 + (|Du|-1)_+)^2 \, dx.
\end{split}
\label{Sobolev3}
\end{equation}%
Let us assume 
\begin{equation}
\Phi (t)=(1+t)^{\gamma -2}t^{2}\qquad \gamma \geq 0
\label{defPhi}
\end{equation}%
from which we deduce
\[
\Phi'(t)=(1+t)^{\gamma -3}t(\gamma t+2) \le \, (\gamma + 2) (1 + t)^{\gamma - 2} t.
\]
With these assumptions \eqref{Sobolev3} reads
\begin{equation}
\begin{split}
& \int_{\Omega }|D(\eta \,G((|Du|-1)_{+})|^{2}\,dx \\
%& \leq \,24\int_{\Omega }|D\eta |^{2}\,(1+\Phi
%(|Du|-1)_{+}h_{2}(|Du|)|Du|^{2})\,dx \\
\le & C \,(\gamma + 2) \int_{\Omega }(\eta^2 + |D \eta|^2) \, \left[ 1+(1+(|Du|-1)_{+})^{\gamma +2}\,h_{2}(1+(|Du|-1)_{+}\right]
\,dx.
\end{split}
\label{Sobolev4}
\end{equation}%
By the Sobolev inequality, there exists a constant $c_{S}$ such that 
\begin{equation}
\left\{ \int_{\Omega }[\eta \,G((|Du|-1)_{+})]^{2^{\ast }}\,dx\right\}
^{2/2^{\ast }}\leq \,c_{S}\,\int_{\Omega }|D(\eta (G(|Du|-1)_{+}))|^{2}\,dx
\label{Sobolev1}
\end{equation}%
where $2^{\ast }=\frac{2n}{n-2}$ if $n>2$ and a number greater than $\frac{2%
}{1-\beta }$ if $n=2$. We apply \textcolor{blue}{\eqref{eq:int1}} with the choice $%
t=(|Du|-1)_{+}$ 
\begin{eqnarray*}
G((|Du|-1)_{+}) &=&1+\int_{0}^{(|Du|-1)_{+}}(1+s)^{\frac{\gamma -2}{2}}s%
\sqrt{h_{1}(1+s)}\,ds \\
&\geq &C_{3}\left[ 1+h_{2}(1+(|Du|-1)_{+})^{\frac{1}{2^{\ast }}}\frac{%
(1+(|Du|-1)_{+})^{\frac{\gamma }{2}+1-\beta }}{\left( \frac{\gamma }{2}%
+1-\beta \right) ^{2}}\right]
\end{eqnarray*}%
thus by \eqref{Sobolev4} we obtain that there exists $c=c(C_{3})>0$ such
that, for all $\gamma \geq 0$, 
\begin{equation}
\begin{split}
& \left\{ \int_{\Omega }\eta ^{2^{\ast }}(1+(1+(|Du|-1)_{+})^{(\gamma
+2-2\beta )\frac{2^{\ast }}{2}}\,h_{2}(1+(|Du|-1)_{+}))\,dx\right\} ^{\frac{2%
}{2^{\ast }}} \\
& \qquad \leq \, c \,\left( \frac{\gamma }{2}+1-\beta \right)
^{4}\,(\gamma + 2)\int_{\Omega} (\eta^2 + |D\eta |^{2}) \,(1+(1+(|Du|-1)_{+})^{\gamma
+2}\\
&\times h_{2}(1+(|Du|-1)_{+})\,dx \\
& \qquad \leq \, c \,\left( \gamma +2\right)^{5}\,\int_{\Omega }(\eta^2 + |D \eta|^2) \, \left[ 1+(1+(|Du|-1)_{+})^{\gamma +2}\,h_{2}(1+(|Du|-1)_{+}\right]
\,dx
\end{split}%
\end{equation}%
where we used once more \eqref{Sobolev3} and \eqref{Sobolev1}. From this point onward, we follow the proofs of \cite[Lemma 6.2]{EMMP} and \cite[Lemma 6.3]{EMMP}. The previous inequality, indeed, is the analogous of (6.10) in \cite{EMMP}. Now, by the same iteration process, we obtain that, for $0 < \rho < R,$ $\overline{B}_R \subset \Omega,$ there exists a positive constant $C$ depending only on $n, L, h_1(1),$ such that
\begin{equation}
\label{PP}
\| 1 + (|Du| - 1)_+ \|^{2 - n \beta}_{L^{\infty}(B_{\rho})} \le \, \frac{C}{(R - \rho)^n} \int_{B_R} (1 + (|Du| - 1)_+)^2 h_2((1 + (|Du| - 1)_+) \, dx.
\end{equation}
As in \cite[Lemma 6.3]{EMMP}, set
\[
V = (1 + (|Du| - 1)_+))^2 h_2((|Du| - 1)_+).
\]
By (F2) and \eqref{PP} there exists $C_{\mu} > 0$ such that
\[
\left \| V\right \|^{\frac{2 - n \beta}{2 - \mu}}_{L^{\infty}(B_{\rho)}}  \le \, \frac{C_{\mu}}{(R - \rho)^n} \int_{B_R} V(x) \, dx.
\]
Moreover, since by \eqref{ab}
\[
\frac{2 - \mu}{2 - n \beta} \left (1 - \frac{1}{\alpha} \right ) < 1,
\] 
from (F4) we can deduce \eqref{PP2}.

%\textcolor{blue}{fino a qui}

%and inserting this in \eqref{(25)Jota} we obtain
%\\
%\begin{equation}
%\begin{split}
%& \int_{\Omega } \eta ^{2} \Phi (|Du|-1)_{+}
%g_{1}(1+(|Du|-1)_{+})\,|D(|Du|-1)_{+}|^{2}\,dx \\
%\le & \int_{\Omega }\eta ^{2}[\Phi (|Du|-1)_{+}+|Du|\Phi ^{\prime }(|Du|-1)_{+}]{%
%g_{1}(1+(|Du|-1)_{+}))}\,|D(|Du|-1)_{+}|^{2}\,dx \\
%\leq & C(1+\gamma)^2  \int_{\Omega} (\eta^2 + |D \eta|^2)   %g_2(1 + (|Du| - 1)_+) (1 + (|Du|-1)_+)^{\gamma+2} \, dx.
%\end{split}
%\label{(26)Jota}
%\end{equation}%
%\textcolor{red}{finire la stima a priori}
\end{proof}

%\section{Lemma on boundedness}

\section{Proof of Theorem \ref{superlinear}}

\label{quattro}

The proof of the theorem is divided in three steps. %approximation, estimates for the boundedness of the gradient and passage to the limit. 
We start by considering, as in \cite{EMMP}, suitable approximations of the functional; in the second step we consider minimizers of these approximating functionals with regular boundary conditions that, in particular satisfy (BSC). This allows us to use the a priori estimates of Section 3. The coercivity of the functional is a crucial property  to perform the passage to the limit in Step 3. We remark that the superlinearity of $f$ follows from assumption (F2): in fact it implies that $h_1(t)\ge h_1(t_0)t^{-1}$ for $t\ge t_0$, then 
 there exists $m>0$ such that
 \begin{equation}
\label{remsuperlinarity}
f(\xi)\ge m|\xi|\log|\xi|
\end{equation}
for $\xi$ sufficiently large (see Lemma 7.2 in \cite{EMMP}).

\bigskip

\textsc{step 1: approximation.} In \cite{EMMP} (see the proof of Theorem 2.1) it has been proved that there exists  a sequence $f_k\in\mathcal{C}^2(\R^n)$ of locally uniformly convex functions such that 
\begin{enumerate}
    \item $f_k$ satisfies (F1)-(F4) with $2h_2$ instead of $h_2$,  constants 
$C_{1}$ and $C_{2}$ independent of $k$;
\item $f_k$ uniformly converges to $f$ on compact sets;
\item for every $\delta >0$ and for every $%
k$ sufficiently large 
\begin{equation}
\label{delta}
f(\xi)\leq
\begin{cases}
f_k(\xi)+\delta & \text{ if $|\xi |\leq t_0+2$}\\
f_k(\xi)        & \text{ if $|\xi |> t_0+2$};
\end{cases}
\end{equation}
\item for every  $\xi\in\R^n$
$$
f(\xi)-1\le f_k(\xi)\le f(\xi)+|\xi|+1.
$$
\end{enumerate}
We observe that, thanks to \eqref{remsuperlinarity}, the functions $f_k$
are superlinear.

Let $u_\epsilon$ a mollification of $u$ on $B_R$. Let $v_{k,\epsilon}$ be the minimizer of 
\begin{equation}
    \label{funz-k-g}
\int_{B_R} f_k(D v)+g(x,v)\, dx
\end{equation}
such that $v=u_\epsilon$ on $\partial B_R$. We observe that $u_\epsilon\in\mathcal{C}^\infty(B_R)$ and hence (see \cite{Miranda, Giusti}) it fulfills the (BSC) on $\partial B_R$.

\vspace{5mm}

\textsc{step 2: boundedness of the gradients.} In this step we are going to prove that $v_{k.\epsilon}\in W^{1,\infty}(B_R)$ for $R$ sufficiently small.

%\begin{lem}\label{LemmaBSC}
%Let $v_{k,\epsilon}$ be the minimizer of functional \eqref{funz-k-g} with boundary datum $u_\epsilon\in \mathcal{C}^\infty(B_R)$.
%Assume that
%\begin{enumerate}
 %   \item[a)] $f$ satisfies assumption \eqref{H}
%    \item[c)] $|g_u(x,u)|_\infty\le L$
%    \item[d)] there exists
%a positive constant $K$ such that
%\[
%\forall x,y\in\R^n,\, \forall u,v\in\R\quad %v\ge u+K|y-x|
%\Rightarrow g^+_v(y,v)\ge g^+_v(x,u)
%\]
%where $g^+_v$ denotes the right derivative of
%$g$ with respect to the second variable. 
%\end{enumerate}
% Moreover assume that $R$ is sufficiently small, then $v_{k.\epsilon}\in W^{1,\infty}(B_R)$.
%\end{lem}

%\begin{proof}\ 
First of all we recall that, since $u_{\varepsilon}$ satisfies the (BSC) on $B_R$,
for every $z \in \partial B_R,$ there exists $\kappa_z^-$ and $\kappa_z^+$ such that
\begin{equation}\label{BSC}
\kappa_z^- (x - z) + u_{\varepsilon}(z) \le \, u_{\varepsilon}(x) \le \kappa_z^+ (x - z) + u_{\varepsilon}(z) \qquad \forall x \in \partial \Omega,
\end{equation}
see \cite{Miranda, Giusti}. 
Our aim is to construct lower and upper Lipschitz barriers for the boundary datum. We follow the ideas used in \cite{FT} and \cite{GT}, remarking the fact that here we are in a slightly different set of assumptions.

We recall that by Proposition \ref{polare}  we have 
that $f_k^*$  is defined in $\R^n$ and superlinear.

Let us fix $z\in \partial B_R$ and let $\kappa_z^-$ as in the left hand side of \eqref{BSC}. We consider the set
\[
\left \{\frac{n}{L} f_k^*\left (\frac{L}{n}x \right) - \kappa_z^- \cdot x - c \le 0\right \} =  \Omega_{\kappa_z^-, c}.
\]
We observe that for $c$ sufficiently large $\Omega_{\kappa_z^-, c}$ is not empty and convex. The superlinearity of $f^*_k$ implies that it is bounded and the fact that $f^*_k$ is finite for every $x\in\R^n$ implies that 
\begin{equation}
    \lim_{c\to + \infty} \min\{|x|:\ x\in\partial \Omega_{\kappa_z^-, c}\}=+\infty.
\end{equation}
Moreover,  by Proposition \ref{polare} v) it  follows that, for $c$ sufficiently large,  $\partial \Omega_{\kappa_z^-, c}$ is $\mathcal{C}^2$
so that 
we can perform the same computations as in Step 2 of the proof of Theorem 4.5 in \cite{GT} and we can show that the principal curvatures of $\partial \Omega_{\kappa_z^-, c}$ at every point $x$ are less or equal than
  
%The identity $Df_k(Df_k^*(\xi))=\xi$ implies that $D^2f_k^*=(D^2f_k)^{-1}$ and 

\begin{equation}\label{curvatura}
\frac{|D^2 f^*_k(\frac{L}{n} x)|}{|D f^*_k
(\frac{L}{n} x)|} \le \frac{1}{h_1(|Df^*_k(\frac{L}{n}x)|)|Df^*_k(\frac{L}{n}x)|},
\end{equation}
where we have also used assumption (F1).

Now we fix $\frac{L}{n}x\in\R^n\setminus B_{s_0},$ where $s_0$ is given by Proposition \ref{polare} v), and we define
\[
\varphi(t)=f_k^*\left(t\frac{x}{|x|}\right)\quad\text{ for }\, t\ge s_{0};
\]
then we obtain
\[
\varphi'(t)=Df_k^*\left(t\frac{x}{|x|}\right)\cdot \frac{x}{|x|}  \quad\text{ for }\, t\ge s_{0}
\]
and, by using once more Proposition \ref{polare} v) and assumption (F1)
\[
\varphi''(t)=D^2f_k^*\left(t\frac{x}{|x|}\right)\frac{x}{|x|}\cdot \frac{x}{|x|}\ge \frac{1}{h_2(t)} \quad\text{ for }\, t\ge s_{0}.
\]
It follows that, for $t=\frac{L}{n}|x|$, there exists a non negative constant $C$ such that
\[
\left|Df_k^*\left(\frac{L}{n}x\right)\right|\ge \varphi'\left(\frac{L}{n}|x|\right)\ge \int_{s_0}^{\frac{L}{n}|x|}\frac{1}{h_2(\tau)}\,d\tau +C
\]
and the last term goes to $+\infty$ as $|x|\to +\infty$.
Assumption (F2) implies that there exists $\bar t$ such that $h_1( t)t\ge \delta>0$ for every $t\ge \bar t$. It follows that, if $c$ is sufficiently large the principal curvatures of $\Omega_{\kappa_z^-, c}$ are less or equal to $\frac{1}{\delta}$.

Let now $R<\delta$ and $\nu$ be the normal vector to $\partial B_R$ in $z.$ Thus, there exists $x_z \in \partial \Omega_{\kappa_z^-, c}$ such that its normal vector is exactly $\nu$. 

Let us consider the function
\[
v_z(x)=\frac{n}{L} f^*_k\left (\frac{L}{n} (x - (z - x_z)) \right) + u_{\varepsilon}(z) - \frac{n}{L} f^*_k\left (\frac{L}{n} x_z \right) 
\]

We define 
\begin{equation}
{\tilde\Omega_{\kappa_z^-,c}}=\left\{v_z(x) -\kappa^-_z\cdot(x-(z-x_z))- u_{\varepsilon}(z) + \frac{n}{L} f^*_k\left (\frac{L}{n} x_z \right)-c\le 0\right\} \end{equation}
 and obviously ${\tilde\Omega_{\kappa_z^-,c}}= \Omega_{\kappa_z^-,c} +(z-x_z)$ so that the curvature of $\partial{\tilde\Omega_{\kappa_z^-,c}}$ in $z$ is the same of $\partial\Omega_{\kappa_z^-,c}$ in $x_z$.

Since $R<\delta,$ we have that $B_R\subset {\tilde\Omega_{\kappa_z^-,c}}$ and $z\in \partial B_R\cap \partial{\tilde\Omega_{\kappa_z^-,c}}$. Moreover we remark that
 \begin{equation}
{\tilde\Omega_{\kappa_z^-,c}} =\{v_z(x) \le \kappa^-_z\cdot(x-z)+ u_{\varepsilon}(z)\}  
 \end{equation}
so that we can apply the comparison principle in \cite[Theorem 2.4]{FT} and in \cite[Theorem 2.4]{GT} between the minimizer $v_{k,\varepsilon}$ and the function $v_z$ and conclude that $v_{k,\varepsilon}(x)\ge v_z(x)$ a.e. in $B_R$. 
The construction of the lower barrier is completed considering every $z\in\partial B_R$ and defining
\begin{equation}
\ell^-(x)= \sup_{z\in\partial B_R} v_z(x).
\end{equation}
Repeating an analogous construction we can construct also the upper barrier $\ell^+$.

Remarking that $\ell^\pm$ are Lipschitz in $B_R$ and arguing as in  \cite[Theorem 5.2]{MTrado} and  in \cite[Theorem 4.6]{GT}, we conclude the proof of this step.
%\end{proof} 

\vspace{5mm}

\textsc{step 3: passage to the limit.}
 We can apply Lemma \ref{Lemmabound} with $f_-(\xi)=f(\xi)-1$ and $f_+(\xi)= f(\xi)+|\xi|+1$ to deduce that there exists a constant $\tilde M$ such that $\|v_{k,\epsilon}\|_{L^\infty(B_R)}\le \tilde M$ for every $k$ and $\epsilon$.
Hence assumptions (G1) and (G2) imply that there exists a constant $\tilde K$ such that that $\|g(x,u_{\epsilon})\|_{L^{1}(B_R)} \le \tilde K$ and $\|g(x,v_{k,\epsilon})\|_{L^{1}(B_R)} \le \tilde K$.
Step 2 implies that $v_{k,\epsilon}\in W^{1,\infty}(B_R)$
and from Theorem \ref{stimapriori} we get the estimate
\[
\|D v_{k,\epsilon}\|_{L^\infty(B_\rho; \mathbb{R}^n)}\le C \, \left(\frac{1}{(R-\rho)^n}\int_{B_R} \{1+f_k(D v_{k,\epsilon})\}\,dx\right)^\theta
\]
where the constant $C$ does not depend on $k$ and $\epsilon$. Adding and subtracting  $g(x,v_{k,\epsilon})$ and using the minimality of $v_{k,\epsilon}$ we obtain
\[
\begin{split}
\|Dv_{k,\epsilon}\|_{L^\infty(B_\rho; \mathbb{R}^n)}
 \le & C \left(\frac{1}{(R-\rho)^n}\int_{B_R} \{1+f_k(D v_{k,\epsilon})+g(x,v_{k,\epsilon}) \} \,dx \right.\\
 & \left. -\frac{1}{(R-\rho)^n}\int_{B_R}  g(x,v_{k,\epsilon})\,dx\right)^\theta \\
 \le & C \, \left(\frac{1}{(R-\rho)^n}\int_{B_R} \{1+f_k(D u_{\epsilon})+g(x,u_{\epsilon})\}\,dx +\frac{\tilde K}{(R-\rho)^n}\right)^\theta .
\end{split}
\]
Therefore
\[
\limsup_{k\to +\infty}\|D v_{k,\epsilon}\|_{L^\infty(B_\rho; \mathbb{R}^n)}
 \le M_\varepsilon
 \]
where 
 \[
 M_\varepsilon= C
 \left[\frac{1}{(R-\rho)^n}\left(\int_{B_R} \{1+f(D u_{\epsilon})+g(x,u_{\epsilon})\} \,dx +\tilde K\right)\right]^\theta .
\]
The sequence $v_{\varepsilon ,k}$ is bounded in $W^{1,\infty }(B_{\rho })$ uniformly
with respect to $k$, then there exists a subsequence $k_{j}\rightarrow \infty$,
such that $\{v_{\varepsilon ,k_{j}}\}$ is weakly$^{\ast }$ convergent in $W^{1,\infty }(B_{\rho })$. 
Now we fix a sequence $\rho_j\rightarrow R$ and, 
by a diagonalization argument, we extract a subsequence, that we still denote by $\{v_{\varepsilon ,k_{j}}\}$, weakly$^*$ converging to
$\bar{v}_{\varepsilon }$ in $W^{1,\infty }(B_{\rho })$ for every $\rho<R$. Recall that $\{v_{\varepsilon ,k_{j}}\}\subset
u_\varepsilon+W_{0}^{1,1}(B_{R})$.
Moreover for every $\rho <R$
\begin{equation}
\Vert D\bar{v}_{\varepsilon }\Vert _{L^{\infty }(B_{\rho}; \mathbb{R}^n)}\leq
M_\varepsilon.  \label{normepsilon}
\end{equation}%
The next step is to prove that $v_{\varepsilon ,k_{j}}$ weakly converges to $\bar{v}_{\varepsilon }$
in $W^{1,1}(B_{R})$ so that $\bar{v}_{\varepsilon }\in u_\varepsilon+W_{0}^{1,1}(B_{R})$.
Indeed by  the minimality of $v_{\varepsilon ,k_{j}}$,
as $j\rightarrow \infty $ we have 
\begin{equation*}
\begin{split}
\int_{B_{R}}f(Dv_{\varepsilon ,k_{j}})\,dx
& \leq \int_{B_{R}}1+f_{k_{j}}(Dv_{\varepsilon ,k_{j}})
 +g(x,v_{\varepsilon ,k_{j}})\,dx - \int_{B_{R}}g(x,v_{\varepsilon ,k_{j}})\,dx \\
& \leq \int_{B_{R}}{f}_{k_{j}}(Du_{\varepsilon})
+g(x,u_{\varepsilon})\,dx+(\tilde K +1)  \\
&\rightarrow \int_{B_{R}}f(Du_{\varepsilon
})+g(x,u_{\varepsilon})\,dx+(\tilde K+1).
\end{split}
\end{equation*}%
The superlinearity of $f$ and  de la Vall\'{e}e-Poussin Theorem imply that we can choose the sequence $k_{j}$ such
that $Dv_{\varepsilon ,k_{j}}\rightharpoonup D\bar{v}_{\varepsilon }$ in $%
L^{1}(B_{R}; \mathbb{R}^n)$ and then $(v_{\varepsilon ,k_{j}}-u_{\varepsilon
})\rightharpoonup (\bar{v}_{\varepsilon }-u_{\varepsilon })\in
W_{0}^{1,1}(B_{R})$.

On the other hand, by  the minimality of $v_{\varepsilon,k_{j}}$ 
\begin{equation*}
\begin{split}
\int_{B_{R }}&f(Dv_{\varepsilon ,k_{j}}) + g(x,v_{\varepsilon ,k_{j}})\,dx\\ 
= &\int_{B_{R}}{f}_{k_{j}}(Dv_{\varepsilon ,k_{j}})+g(x,v_{\varepsilon ,k_{j}})\,dx+\int_{B_{R}}(f(Dv_{\varepsilon ,k_{j}})-{f}_{k_{j}}(Dv_{\varepsilon ,k_{j}})) \, dx \\
 \leq & \int_{B_{R}} f_{k_{j}}(Du_{\varepsilon })+g(x,u_{\varepsilon})\,dx+\int_{B_{R}}(f(Dv_{\varepsilon ,k_{j}})-{f}_{k_{j}}(Dv_{\varepsilon ,k_{j}})) \, dx.
\end{split}%
\end{equation*}%
By 
\eqref{delta} for every $\delta$ there exists $\bar k$ such that for every $k_j>\bar k$ 
\begin{equation*}
\int_{B_{R }}f(Dv_{\varepsilon ,k_{j}}) + g(x,v_{\varepsilon ,k_{j}})\,dx\leq \int_{B_{R}} f_{k_{j}}(Du_{\varepsilon })+g(x,u_{\varepsilon})\,dx+\delta |B_{R}|.
\end{equation*}

By lower semicontinuity in $W^{1,1}(B_{R})$, passing to the limit for $%
j\rightarrow \infty $, we get 
\begin{equation*}
\begin{split}
\int_{B_{R }}&f(D\bar{v}_{\varepsilon }) +g(x,\bar v_{\varepsilon })\,dx
 \leq \liminf_{j\rightarrow
\infty }\int_{B_{R }}f(Dv_{\varepsilon ,k_{j}}) + g(x,v_{\varepsilon ,k_{j}})\,dx\\
&\leq
\lim_{j\rightarrow \infty }\int_{B_{R}}f_{k_{j}}(Du_{\varepsilon })+g(x,u_{\varepsilon})\,dx+\delta |B_{R}| \\
& =\int_{B_{R}}f(Du_{\varepsilon })+g(x,u_{\varepsilon})\,dx+\delta |B_{R}|
\end{split}%
\end{equation*}%
for every $\delta >0$ and then for  $
\delta \rightarrow 0$ 
\begin{equation}\label{Jensen}
\int_{B_{R}}f(D\bar{v}_{\varepsilon })+g(x,\bar v_{\varepsilon })\,dx\leq \int_{B_{R}}f(Du_{\varepsilon})+g(x,u_{\varepsilon})\,dx.
\end{equation}
We observe that, thanks to Jensen's inequality and the Dominated Convergence Theorem (see \cite{mar96} and \cite[Lemma 7.1]{EMMP}),  
\begin{equation}\label{Jensen2}
\lim_{\epsilon\to 0} \int_{B_{R}}f(Du_{\varepsilon})+g(x,u_{\varepsilon})\,dx=\int_{B_{R}}f(Du)+g(x,u)\,dx
\end{equation}
and hence the right hand side of \eqref{Jensen} is uniformly bounded w.r.t. $\epsilon$.
We apply once more de la Vall\'{e}e-Poussin Theorem to extract  a sequence $\varepsilon _{j}\rightarrow 0$ such that $\bar{%
v}_{\varepsilon _{j}}-u_{\varepsilon _{j}}\rightharpoonup \bar{v}-u$ in $%
W_{0}^{1,1}(B_{R})$. By the lower semicontinuity of the functional, \eqref{Jensen} and \eqref{Jensen2}
\begin{equation}
\begin{split}
& \int_{B_{R}}f(D\bar{v})+g(x,\bar{v})\,dx
 \leq \liminf_{j\rightarrow \infty }\int_{B_{R}}f(D\bar{v}_{\varepsilon _{j}})
   +g(x,\bar v_{\varepsilon _{j}})\,dx\\
&\leq \lim_{j\rightarrow \infty
}\int_{B_{R}}f(Du_{\epsilon _{j}})+g(x,u_{\varepsilon_j})\,dx
=\int_{B_{R}}f(Du)+g(x,u)\,dx.  \label{eq:min}
\end{split}
\end{equation}%
Then $\bar{v}$ is another minimizer for \eqref{funz-modello} with $\Omega
=B_{R} $. Moreover from \eqref{normepsilon} 
 we can also assume that  $\{\bar{v}_{\varepsilon _{j}}\}_{j}$ is weakly$
^{\ast }$ convergent to $\bar{v}$ in $W^{1,\infty }(B_{\rho })$ for every $%
0<\rho <R$.
Therefore, thanks to \eqref{normepsilon} 
 and \eqref{Jensen2},   we have that  for every $0<\rho<R$  
\begin{equation}
\begin{split}
&\Vert D\bar{v}\Vert _{L^{\infty}(B_{\rho }; \mathbb{R}^n)} \leq \,\liminf_{j\rightarrow
\infty }\Vert D\bar{v}_{\varepsilon _{j}}\Vert _{L^{\infty }(B_{\rho }; \mathbb{R}^n)} \\
& \leq
\lim_{j\rightarrow \infty }\,C\,\left\{ \frac{1}{(R-\rho )^{n}}
\left (\int_{B_{R}}1+f(Du_{\varepsilon _{j}})+g(x,u_{\epsilon_j})\,dx +\tilde K \right )\right\}^{\theta } \\
& =\,C\,\left\{ \frac{1}{(R-\rho )^{n}} \left (\int_{B_{R}} f(Du)+g(x,u)\,dx +\kappa \right )\right\}
^{\theta },
\end{split}
\label{eq:lips}
\end{equation}%
where $\kappa = \tilde{K} + |B_R|.$
\\
Since $\bar{v}$ and $u$ are two different minimizers of $F$ in $B_{R}$ and $f(\xi )$ is strictly convex for $|\xi |>t_{0}$, by proceeding as in 
\cite{EMM} it is possible to prove that the set 
\begin{equation*}
E_{0}:=\left\{ x\in B_{R}:\left\vert \frac{Du(x)+D\bar{v}(x)}{2}\right\vert
>t_{0}\right\}\cap\{Du\neq D\bar{v}\}.
\end{equation*}%
has zero measure. Therefore 
\begin{equation*}
\Vert Du\Vert _{L^{\infty }(B_{\rho }; \mathbb{R}^n)}\leq \,\Vert Du+D\bar{v}\Vert
_{L^{\infty }(B_{\rho }; \mathbb{R}^n)}+\Vert D\bar{v}\Vert _{L^{\infty }(B_{\rho }; \mathbb{R}^n)}\leq
\,2t_{0}+\Vert D\bar{v}\Vert _{L^{\infty }(B_{\rho }; \mathbb{R}^n)}.
\end{equation*}

\section{Some additional results}
\label{cinque}

We conclude by presenting some additional results related to Theorem \ref{superlinear}. We start by the following theorem which is a slightly more general version of Theorem \ref{superlinear}, where the global Lipschitz continuity of $g,$ namely assumption (G1), is replaced by the following local Lipschitzianity
\begin{enumerate}
    \item[(G1)'] for every $M>0$, there exists $L(M)$ such that  $|g(x,\eta_1)-g(x,\eta_2)|\le L(M)|\eta_1-\eta_2|$ for a.e. in $x\in\Omega$ and for every $\eta_1,\eta_2\in [-M,M]$.
\end{enumerate}
In this case it turns out that  the constant $C$ in \eqref{infinity nu} depends also on
$\|u\|_{L^{\infty}(B_R)}$.

\begin{theo}
\label{uLocBounded} Let $u\in W_{loc}^{1,1}(\Omega)\cap L_{loc}^\infty(\Omega)$ be a local minimizer of the functional \eqref{funz-modello}. Suppose that $f$ satisfies the growth assumptions \textnormal{(F1)}--\textnormal{(F4)}, with the parameters $\alpha $, $\beta $, $\mu $ related by the
condition \eqref{ab}.
%\begin{equation}
%{2-\mu -\alpha (n\beta -\mu )}>0\,.  %\label{ab}
%\end{equation}%
Assume moreover that $g$ fulfills assumptions \textnormal{(G1)'}-\textnormal{(G2)}-\textnormal{(G3)}-\textnormal{(G4)}.
\\
Then $u$  is locally Lipschitz continuous in $\Omega $ and it satisfies   estimate \eqref{infinity nu} as in Theorem {\rm \ref{superlinear}} where in this case the constant $C$ depends also on $\|u\|_{L^{\infty}(B_R)}.$
\end{theo}

\begin{proof}\  The result is a straighfoward consequence of Theorem \ref{superlinear}: in fact it is sufficient to consider $\|u\|_{L^{\infty}(B_R)}$ with $B_R$ instead of $\Omega$ and to observe that $g$ satisfies (G1)' with $M = \|u\|_{L^{\infty}(B_R)}.$
\end{proof}

\bigskip

The next theorem is obtained by considering  functionals of type \eqref{funz-modello} in the space $u_0+W^{1,1}_0(\Omega),$ where $u_0$ is a fixed boundary datum.

\begin{theo}
\label{uBounded} Let $u_0\in W^{1,1}(\Omega)\cap L^{\infty}(\Omega)$ and $u $ be a minimizer of the functional \eqref{funz-modello} in the class $u_0+W^{1,1}_0(\Omega)$. Suppose that $f$ satisfies the growth assumptions \textnormal{(F1)}--\textnormal{(F4)}, with the parameters $\alpha $, $\beta $, $\mu $ related by the
condition \eqref{ab}.
%\begin{equation}
%{2-\mu -\alpha (n\beta -\mu )}>0\,.  %\label{ab}
%\end{equation}%
Assume moreover that $g$ fulfills assumptions \textnormal{(G1)'}-\textnormal{(G2)}-\textnormal{(G3)}-\textnormal{(G4)}.
\\
Then $u$  is locally Lipschitz continuous in $\Omega $ and it satisfies   estimate \eqref{infinity nu} as in Theorem {\rm \ref{superlinear}} where, this time, the constants $C$ and $\kappa$ depend also on $\|u_0\|_{L^\infty(\Omega)}$.
\end{theo}

\begin{proof}\ We apply Lemma \ref{comparison} to get an $L^{\infty}$ bound for the minimizer and we proceed then as in the previous theorem. 
\end{proof}

\bigskip

\begin{rem}
We notice that any function $g(x,u)$ such that $g(x,u)=g(u)$ satisfying assumption (G3), fulfills assumption (G1)', (G2) and (G4), therefore the only assumption required, in this case, is the convexity.
\\
On the other hand, assumption (G1)' allows us to consider also significant cases for applications. For instance we can deal with functionals modelling the elastoplastic torsion, where
$$g(x,u)=(\lambda u-a(x))u$$
with $a(x)\in W^{1,\infty}(\Omega)$ and $\lambda >0$, or the
reconstruction of an image $u$ from a degraded data $a(x)$, where
$$g(x,u)=|a(x)-\lambda u|^2, \, \,\,\, \, \,a(x) \in C^1(\bar \Omega),\, \lambda\in \R.$$
\end{rem}

We conclude this section by considering the case of radially symmetric lagrangian $f(\xi) = h(|\xi|),$ for a given function $h$. In this case, as we already remarked, condition \eqref{ab} is always satisfied. 

\begin{theo}
\label{radial-symmetric}
  Let $u\in W_{loc}^{1,1}(\Omega)\cap L_{loc}^\infty(\Omega)$ be a local minimizer of the functional \eqref{funz-modello}. Suppose that $f(\xi)=h(|\xi|)$ where $h$ is non negative, convex, increasing, superlinear and $h\in \mathcal{C}([0,+\infty))\cap\mathcal{C}^2([t_0,+\infty))$ for a suitable $t_0>0$. We also assume that there exist $\mu\in [0,1]$, $\beta\in \left(0,\frac{2}{n} \right)$ and a positive constant $C$ such that, for every $t\ge t_0$
  \begin{enumerate}
  \item[i)] $\displaystyle h''(t)\le \frac{h'(t)}{t}$; 
      \item[ii)] $t \mapsto h'(t)t^{\mu-1}$ is decreasing;
      \item[iii)] $\displaystyle h''(t)\ge \frac{c}{t^{\mu\frac{2}{2^*}+2\beta}},$
  \end{enumerate}
  where, as before, $2^* = \frac{2n}{n-2}$ if $n \ge 3$ while in the case $n=2$, it must be replaced with any fixed positive number greater than $\frac{2}{1 - \beta}$. 
  Assume moreover that $g$ fulfills assumptions \textnormal{(G1)}-\textnormal{(G2)}-\textnormal{(G3)}-\textnormal{(G4)}.
\\
Then $u$  is locally Lipschitz continuous in $\Omega $ and it satisfies   estimate \eqref{infinity nu} as in Theorem {\rm \ref{superlinear}}.
\end{theo}
\begin{proof}\ 
First of all we notice that, following the same notation of assumptions (F1)--(F4) and recalling i) and  \cite[equation (3.3)]{MP}, we have that (F1) holds with
\[
h_1(t)=h''(t)\quad\text{and}\quad h_2(t)=\frac{h'(t)}{t}.
\]
We remark that, as for Lemma \ref{g1eg2}, also Theorem \ref{stimapriori} still holds  assuming, in (F2), $t\mapsto t h_2(t)$ is  increasing instead of $t\mapsto t h_1(t)$ is increasing. The convexity of $h$ implies $t\mapsto t h_2(t)$ is  increasing and the first condition in (F2) is satisfied by ii). 

On the other hand, from assumption ii), we infer the existence of a constant $C > 0$ such that
\[
h'(t) \le \, \frac{C}{t^{\mu - 1}}.
\]
Therefore
\begin{eqnarray*}
[h_2(t)]^{\frac{2}{2^*}} &=& \left [\frac{h'(t)}{t} \right]^{\frac{2}{2^*}} \le \, \left [ \frac{C}{t^{\mu}}\right ]^{\frac{2}{2^*}} \le  \, C^{\frac{2}{2^*}} \, t^{2 \beta} \frac{1}{t^{\mu + \frac{2}{2^*} + 2 \beta}} \stackrel{\textnormal{iii)}}{\le} \frac{C^{\frac{2}{2^*}}}{c} t^{2 \beta}h''(t) \\
&=& C_1 t^{2 \beta} h_1(t)
\end{eqnarray*}
and (F3) holds. Finally, it is sufficient  to show that (F4) holds for $\alpha=1$. 

We have
\begin{eqnarray*}
h(t) - h(t_0) &=& \int_{t_0}^t h'(s)  ds = \int_{t_0}^t \frac{h'(s)}{s^{\mu - 1}} \, s^{\mu - 1}  ds \stackrel{\textnormal{ii)}}{\ge}  h'(t) t^{\mu - 1} \int_{t_0}^t s^{1 - \mu}  ds\\
%= h'(t) t^{\mu - 1} \left [ \frac{s^{2 - \mu}}{2 - \mu}\right ]_{t_0}^t\\
&=& %\frac{1}{2 - \mu} h'(t) t^{\mu -1} \left [ t^{2 - \mu} - t_0^{2 - \mu} \right ] 
 \frac{1}{2 - \mu} h'(t) t - \frac{1}{2 - \mu} h'(t) t^{\mu - 1} t_0^{2 - \mu}\stackrel{\textnormal{ii)}}{\ge} \frac{1}{2 - \mu} h'(t) t - \frac{1}{2 - \mu} h'(t_0) t_0
\end{eqnarray*}
therefore
\[
h'(t) t \le \, (2 - \mu) h(t) - (2 - \mu) h(t_0) + h'(t_0) t_0 \le \, C [h(t) + 1].
\]
Summing up, all the assumptions of Theorem \ref{stimapriori} hold; in particular, being $\alpha = 1$, \eqref{ab} is always satisfied, being equivalent to ask that $\beta < \frac{2}{n}$ so that the a priori estimate holds true. 
It remains to discuss the proof of Theorem \ref{superlinear} in our setting.
\\
\\
 Also in this case we approximate the function $h$ with a sequence of functions $h_k$ satisfying (i), (ii) and (iii) with a constant $c$ independent of $k$. We need to remark that the functions $f_k(\xi)=h_k(|\xi|)$ belong to $ \mathcal{C}^2(\mathbb{R}^n)$ and are locally uniformly convex in $\R^n$.

From now on the proof follows the same ideas of the proof of Theorem \ref{superlinear}. We only underline that the computation of the curvatures of the boundary of the set

\[
\Omega_{\kappa_z^-, c}=\left \{\frac{n}{L} h_k^*\left (\frac{L}{n} |x| \right) - \kappa_z^- \cdot x - c \le 0\right \} 
\]
can be estimated, as in Step 2 of the proof of  \cite[Theorem 4.3]{FT}, by
\begin{equation}\label{curvatura-bis}
\frac{|(h^*_k)''(\frac{L}{n} |x|)|}{|(h^*_k)'
(\frac{L}{n} |x|)|^3} = \frac{1}{h_k''(|(h^*_k)'(\frac{L}{n}|x|)|)|(h^*_k)'(\frac{L}{n}|x|)|^3}.
\end{equation}
At this point, iii) and the fact that $\mu\frac{2}{2^*}+2\beta<3$, yields that
\[
\lim_{t\to+\infty} h''_k (t) t^3 =+\infty
\]
which allows us to infer the existence of $\bar t$ such that $h''_k (t) t^3 \ge \delta>0$ for every $t\ge \bar t$. It follows that if $c$ is sufficiently large, the principal curvatures of $\Omega_{\kappa_z^-, c}$ are less or equal to $\frac{1}{\delta}$ and therefore it is now possible then to conclude as in Step 3 of Theorem \ref{superlinear}.
\end{proof}

\bigskip

\textbf{Data availability}

No datasets were generated or analysed during the current study.

\bigskip

\textbf{Acknowledgments} The authors are indebted to Prof. Giuseppe Mingione for having suggested the problem.

The authors have been partially supported by the Gruppo
Nazionale per l’Analisi Matematica, la Probabilità e le loro Applicazioni (GNAMPA) of the Istituto
Nazio-nale di Alta Matematica (INdAM), through the projects ``Prospettive nelle scienze dei materiali: modelli variazionali, analisi asintotica e omogeneizzazione'' (coordinator E. Zappale) and ``Su alcuni problemi di regolarit\`a del Calcolo delle Variazioni con convessit\`a degenere'' (coordinator F. Giannetti). Moreover M. Eleuteri and S. Perrotta have been partially supported by PRIN 2020 ``Mathematics for industry 4.0 (Math4I4)'' (coordinator P. Ciarletta) and G. Treu has been partially supported by Unipd project DOR2340044.

\bigskip

%%%%%%%%%%%%%%%%%%%%%%%%%%%%%%%%%%%%%%%%%%%%%%%%%%%%%%%%%%%%%%%%%%%%%%%%%%%%%%%%%%%%%%%%%%%%%%%%%%%%%%%%%%%%%%%%%%%%%%%%%%%%%%%%%%%%%%%%%%%%%%%%%%%%%%%%%%%%%%%%%%%%%%%%%%%%%%%%%%%%%%%%%%%%%%%%%%%%%%%%%%%%%%%%%%%%%%%%%%%%%%%%%%%%%%%%%%%%%%%%%%%%%%%%%%%%%%%%%%%%%%%%%%%%%%%%%%%%%%%%%%%%%%%%%%%%%%

\end{document}